\documentclass[11pt]{amsart}
\usepackage{amssymb}
\usepackage{mathtools}
\mathtoolsset{showonlyrefs}

\usepackage{hyperref,color}
\usepackage[usenames,dvipsnames]{xcolor}
\hypersetup{pdftitle={},pdfauthor={},colorlinks=true,linkcolor=MidnightBlue,citecolor=MidnightBlue,urlcolor=MidnightBlue,linktoc=all}

\numberwithin{equation}{section}

\theoremstyle{plain}
\newtheorem{theorem}{Theorem}[section]

\newtheorem{corollary}[theorem]{Corollary}
\newtheorem{lemma}[theorem]{Lemma}

\newtheorem*{theorem*}{Theorem}
\newtheorem*{proposition*}{Proposition}
\newtheorem*{corollary*}{Corollary}
\newtheorem*{lemma*}{Lemma}
\newtheorem*{conjecture*}{Conjecture}

\theoremstyle{definition}
\newtheorem{definition}[theorem]{Definition}

\newtheorem*{definition*}{Definition}
\newtheorem*{example*}{Example}
\newtheorem*{question*}{Question}
\newtheorem*{philosophy*}{Philosophy}

\theoremstyle{remark}
\newtheorem{remark}[theorem]{Remark}

\newtheorem*{remark*}{Remark}

\newcommand\lp{\left(}
\newcommand\rp{\right)}

\newcommand\A{\mathcal{A}}
\newcommand\gal{\mathrm{Gal}}
\newcommand\per{per}

\newcommand{\spe}{\mathrm{Spec}}
\newcommand\du{\vee}
\newcommand\Fp{\Z/p\Z}

\newcommand\Z{\mathbb{Z}}          \newcommand\Q{\mathbb{Q}}                    \newcommand\C{\mathbb{C}}          
\renewcommand\P{\mathbb{P}}                                        
\newcommand\Qb{\overline{\Q}}		

\newcommand\cP{\mathcal{P}}
\newcommand\cO{\mathcal{O}}

\renewcommand\hom{\operatorname{Hom}}

\def\Frob{Frobenian}
\def\P{\cP^0_{\!\A}}
\newcommand\rec{recurrent}
\newcommand\surj{\twoheadrightarrow}

\newcommand\iso{\xrightarrow{\sim}}
\newcommand\ack{\section*{Acknowledgements}}

\begin{document}
\title{A finite analogue of the ring of algebraic numbers}
\author{Julian Rosen}
\email{julianrosen@gmail.com}
\date{\today}
\keywords{Finite periods, Frobenius automorphism, linear recurrence, congruence}
\maketitle

%
%
\begin{abstract}
We construct an analogue of the ring of algebraic numbers, living in a quotient of the product of all finite fields of prime order. We use this ring to deduce some results about linear recurrent sequences.
\end{abstract}

%
%
\section{Introduction}
A period is a complex number given as the integral of an algebraic function over a region defined by algebraic inequalities. The set of all periods is a countable subring of $\C$ containing $\Qb$ (see \cite{Kon01} for an overview of periods). Several recent works (e.g.\ \cite{Jar16d,Kan16,Kan13,Ros18a,Sak19,Zha16}) consider ``finite'' analogues of certain periods (finite multiple zeta values, finite multiple polylogarithms, etc.) living in the ring
\[
\A:=\frac{\prod_p\Fp}{\bigoplus_p\Fp},
\]
which was introduced by Konstsevich (\cite{Kon09}, \S2.2). An element of $\A$ is a prime-indexed sequence $(a_p)_p$, with $a_p\in\Fp$, and two sequences are equal if they agree for all sufficiently large $p$. Every non-zero integer is invertible modulo $p$ for all sufficiently large $p$, so there is a diagonal embedding $\Q\hookrightarrow\A$.

\subsection{Results}
The purpose of this paper is to define a countable $\Q$-subalgebra $\P\subset\A$ that is a finite analogue of $\Qb\subset\C$. This algebra is \emph{not} the integral closure of $\Q$ inside $\A$, which has continuum cardinality.

Our first main result is three equivalent characterizations of $\P$.

\begin{theorem}
\label{th1}
The following subsets of $\A$ are equal.
\begin{enumerate}
\item The set of elements $(a_p\mod p)_p$, where $a_0,a_1,a_2,\ldots\in\Q$ is a recurrent sequence (that is, a sequence satisfying a linear recurrence relation with constant coefficients).
\item The set of elements $\big(g(\phi_p)\mod p\big)_p$, where $L/\Q$ is a finite Galois extension, $g:\gal(L/\Q)\to L$ satisfies $g(\sigma\tau\sigma^{-1})=\sigma(g(\tau))$, and $\phi_p$ is the\footnote{This is independent of the representative of the Frobenius conjugacy class (see \S2).} Frobenius at $p$.

\item The set of $\Q$-linear combinations of matrix coefficients for the $\A$-valued Frobenius automorphism
\[
F_{\!\A}:L\otimes\A\to L\otimes\A,
\]
defined by Definition \ref{deffa}, as $L$ ranges over all number fields.
\end{enumerate}
\end{theorem}
\noindent The equivalence of (1) and (2) is Theorem \ref{thseqA}, and the equivalence of (2) and (3) is Theorem \ref{thmat}.

\begin{definition}
We define $\P\subset\A$ to be the set given by Theorem \ref{th1}.
\end{definition}

The Skolem-Mahler-Lech theorem says that if $(a_n)$ is a \rec{} sequence, the set $\{n:a_n=0\}$ is has finite symmetric difference with a finite union of arithmetic progressions. As a consequence of Theorem \ref{th1}, we obtain an analogue of Skolem-Mahler-Lech for the set of primes $\{p:a_p\equiv 0\!\! \mod p\}$. A set $P$ of primes is called \emph{Frobenian} (cf.\ \cite{Ser16}, \S3.3) if there is a finite Galois extension $L/\Q$ and a union of conjugacy classes $C\subset\gal(L/\Q)$ such that $P$ has finite symmetric difference with the set of rational primes whose Frobenius conjugacy class is in $C$. The Chebotarev density theorem implies that the natrual density of a \Frob{} set exists and is a rational number.

\begin{corollary}
\label{th4}
A set $P$ of primes is \Frob{} if and only if there exists a \rec{} sequence $(a_n)$ such that
\[
P=\{p:a_p\equiv 0\!\!\mod p\}.
\]
\end{corollary}
\noindent Unlike the Skolem-Mahler-Lech Theorem, Corollary \ref{th4} is effective: given the recurrence relation satisfied by $(a_n)$ and a list of initial values, there is a finite algorithm to determine the number field $L$, union of conjugacy classes $C\subset\gal(L/\Q)$, and the finite exceptional set.

We also prove some results about polynomial equations satisfied by elements of $\P$. The first of these results implies that $\P$ is an integral extension of $\Q$.
\begin{theorem}
\label{th2}
Suppose $\alpha\in\P$. Then there exists a non-zero polynomial $f(x)\in\Q[x]$ such that $f(\alpha)=0$, and every such $f(x)$ has a rational root. 
\end{theorem}

\begin{remark}
The Fibonacci sequence $F_n$
is known to satisfy the congruence $F_p\equiv \lp\frac{p}{5}\rp\!\!\mod p$ for every prime $p$, where $ \lp\frac{p}{5}\rp$ is a Legendre symbol. Thus $f(F_p)\equiv 0\!\!\mod p$ for $p\geq 7$, where $f(x)=x^2-1\in\Q[x]$. Theorem \ref{th2} implies that every \rec{} sequence satisfies an analogous identity for some $f$, which necessarily has a rational root.
\end{remark}

We also prove a result about the density of the set of primes $p$ for which $f(a_p)\equiv 0$ mod $p$, when $(a_p)\in\P$ and $f(x)\in\Q[x]$.

\begin{theorem}
\label{th3}
For $f(x)\in\Q[x]$, we have
\begin{equation}
\label{eqineq}
\sup_{(a_p)\in\P}\delta\bigg(\big\{p:f(a_p)\equiv 0\!\!\mod p\big\}\bigg) = \delta\bigg(\big\{p:f\text{ has a root mod } p\big\}\bigg),
\end{equation}
where $\delta$ denotes natural density. Moreover if $f(x)$ has no rational roots, then there is no element of $\P$ realizing the supremum.
\end{theorem}

In \S\ref{secfrob} we explain the analogy between $\cP_{\!\A}^0\subset\A$ and $\Qb\subset\C$, and the relationship with periods.

%
%
\section{Functions on a Galois group}
\label{secfungal}
Let $L/\Q$ be a finite Galois extension, with ring of integers $\cO_L$ and Galois group $\Gamma:=\gal(L/\Q)$.
\begin{definition}[\cite{Ros17}, \S2]
We define $A(L)$ to be the set of functions $g:\Gamma\to L$ satisfying 
\begin{equation}
\label{eqg}
g(\sigma\tau\sigma^{-1})=\sigma(g(\tau))
\end{equation}
 for all $\sigma$, $\tau\in \Gamma$, which is a commutative $\Q$-algebra under pointwise addition and multiplication.
 \end{definition}
 
For $g\in A(L)$, let $p$ be a rational prime unramified in $L$ that is coprime to the denominators of all values of $g$. Let $\mathfrak{P}$ be a prime of $L$ over $p$, with Frobenius element $\phi_\mathfrak{P}\in \Gamma$. It follows from \eqref{eqg} that the residue class
\begin{equation}
\label{eqres}
g(\phi_\mathfrak{P})\mod \mathfrak{P}
\end{equation}
is fixed by $\phi_\mathfrak{P}$, so \eqref{eqres} is an element of $\Fp\subset \cO_L/\mathfrak{P}$. It can be checked that the value of $g(\phi_\mathfrak{P})$ mod $\mathfrak{P}$ is independent of the choice of $\mathfrak{P}|p$ (see \cite{Ros17}, \S4), and we write $g(\phi_p)$ mod $p$ for this residue class in $\Fp$. We leave $g(\phi_p)$ mod $p$ undefined for the finitely many primes that are either ramified in $L$ or are not coprime to the denominators of $g$.

The following result gives equivalence of conditions (1) and (2) in the statement of Theorem \ref{th1}.
\begin{theorem}
\label{thseqA}
An element of $\A$ has the form $(a_p\mod p)_p$ for some recurrent sequence $(a_n)$ if and only if that element of $\A$ can be written $\big(g(\phi_p)\mod p\big)$ for some finite Galois extension $L/\Q$ and some $g\in A(L)$.
\end{theorem}
\begin{proof}
($\Longrightarrow$) Let $(a_n)$ be a \rec{} sequence. Then there exist column vectors $u$, $v$ and an invertible matrix $M$, with entries in $\Q$, such that
\[
a_n=u^T M^nv
\]
for all $n\in\Z$. There is a Jordan-Chevalley decomposition
\[
M=M_{ss}M_u,
\]
where $M_{ss}$ is semi-simple, $M_u$ is unipotent, and $M_{ss}$ commutes with $M_u$. For every prime $p$ larger than the size of $M_u$ that is coprime to all denominators appearing in $M_u$, the $p$-th power $M_u^p$ is congruent to the identity matrix modulo $p$, and if in addition $p$ is coprime to denominators appearing in $u$ and $v$, then
\begin{equation}
\label{eqss}
a_p\equiv u^TM_{ss}^pv\mod p.
\end{equation}

Let $L$ be a finite Galois extension of $\Q$ over which $M_{ss}$ diagonalizes, let $\lambda_1,\ldots,\lambda_k\in L$ be the eigenvalues of $M_{ss}$, and write $\Gamma=\gal(L/\Q)$. Using the Jordan normal form of $M_{ss}$, it follows from \eqref{eqss} that there are elements $b_1,\ldots,b_k\in L$ such that
\[
a_p\equiv \sum_i b_i \lambda_i^p \mod p,
\]
and $\Gamma$ permutes the pairs $b_i$, $\lambda_i$, i.e.\ the element
\[
\alpha:=\sum_i b_i\otimes\lambda_i\in L\otimes_\Q L
\]
is invariant under the diagonal action of $\Gamma$. There is a canonical isomorphism
\begin{align*}
\varphi:L\otimes_\Q L&\to \hom(\Gamma,L),\\
x\otimes y&\mapsto \bigg(\sigma\mapsto x\sigma(y)\bigg),
\end{align*}
taking the $\Gamma$-invariant elements of $L\otimes L$ to $A(L)$, and we let $g=\varphi(\alpha)\in A(L)$. If $p$ is a rational prime unramified in $L$ coprime to every denominator of the values of $g$, then for every prime $\mathfrak{P}$ of $L$ over $p$,
\begin{align*}
g(\phi_{\mathfrak{P}})&=\sum_i b_i \phi_{\mathfrak{P}}(\lambda_i)\\
&\equiv\sum_i b_i \lambda_i^p\mod \mathfrak{P}\\
&\equiv a_p\mod \mathfrak{P}.
\end{align*}
Thuse we have $(a_p\mod p)=(g(\phi_p)\mod p)$.

($\Longleftarrow$) Suppose $g\in A(L)$ is given, and let
\begin{equation}
\varphi^{-1}(g)=\sum b_i\otimes \lambda_i\in (L\otimes L)^\Gamma,
\end{equation}
where we may choose $b_i$, $\lambda_i$ such that the pairs $(b_i,\lambda_i)$ are permuted by $\Gamma$. Then the sequence
\[
a_n:= \sum_i b_i \lambda_i^n
\]
is \rec{}, and takes values in $\Q$ because the pairs $(b_i,\lambda_i)$ are permuted by $\Gamma$. By the computation above, we see that
\[
a_p\equiv g(\phi_p)\mod p
\]
for all sufficiently large $p$. This completes the proof.
\end{proof}

\begin{remark}
\label{remper}
Let $L/\Q$ be a finite Galois extension. In the language of motives, the ring $A(L)$ defined in \S\ref{secfungal} is the ring of de Rham motivic periods of $\spe\,L$ (see \cite{Bro14}, \S1.2, and \cite{Ros17}, \S5). There is a ring homomorphism
\begin{align*}
\per_{\!\A}:A(L)&\to\A,\\
g&\mapsto \big(g(\phi_p)\mod p\big)_p,
\end{align*}
which is an example of an \emph{$\A$-valued period map} (see \cite{Ros18a}, \S5).
\end{remark}

%
%
\section{Proofs of the theorems}
\label{secproofs}
\noindent In this section we prove Corollary \ref{th4}, and Theorems \ref{th2} and \ref{th3}.
\bigskip

\begin{proof}[Proof of Corollary \ref{th4}]
Suppose $(a_n)$ is a \rec{} sequence. Let $L/\Q$ and $g\in A(L)$ be as in the statement of Theorem \ref{thseqA}. Then for all primes $p$ unramified in $L$ coprime to the numerators and denominators of all non-zero values of $L$ and all $\mathfrak{P}|p$, we have
\[
a_p\equiv 0\mod p\Leftrightarrow g(\phi_\mathfrak{P})=0.
\]
So we may take $C=\{\sigma\in\gal(L/\Q):g(\sigma)=0\}$, which is a union of conjugacy classes by \eqref{eqg}.

Conversely, suppose $L/\Q$ and $C\subset\gal(L/\Q)$ are given. Let $g\in A(L)$ be the characteristic function of $C$, and let $a_n$ be a \rec{} sequence such that
\[
a_p\equiv g(\phi_p)\mod p
\]
for all but finitely many $p$ (which exists by Theorem \ref{thseqA}). Then $\{p:a_p\equiv 0\mod p\}$ coincides with $\{p:\phi_p\subset C\}$ up to a finite set. We can multiply the sequence $(a_n)$ through by a constant rational number to modify $\{p:a_p\equiv 0\mod p\}$ by any finite set. This completes the proof.
\end{proof}

\begin{proof}[Proof of Theorem \ref{th2}]
Suppose $(a_p)_p\in\P$ is given. By Theorem \ref{thseqA}, we can find $L/\Q$ and $g\in A(L)$ such that $a_p\equiv g(\phi_p)\mod p$ for all sufficiently large $p$. Since $A(L)$ is a finite-dimensional $\Q$-algebra, there is a non-zero $f(x)\in\Q[x]$ such that $f(g)=0$, which implies
\begin{equation}
\label{eq0}
f(a_p)\equiv f(g(\phi_p))\equiv  0\mod p
\end{equation}
for all sufficiently large $p$. We can scale $f(x)$ by a rational constant to make \eqref{eq0} hold for all $p$.

Now suppose we are given $f(x)\in\Q[x]$ with $f(a_p)\equiv 0$ mod $p$ for all $p$. There are infinitely many primes $p$ that split completely in $L$, and for all but finitely many of these $p$, we have
\begin{equation}
\label{zerop}
f(a_p)\equiv f(g(1))\equiv 0\mod p,
\end{equation}
where $1\in\gal(L/\Q)$ is the identity element. Since \eqref{zerop} holds for arbitrarily large $p$, it follows that $f(g(1))=0$. Finally, \eqref{eqg} implies that $g(1)\in\Q$, so we conclude $f(x)$ has a rational root.
\end{proof}

Before proving Theorem \ref{th3}, we need some preliminary results. Suppose $f(x)\in\Q[x]$ is monic, let $L/\Q$ be a finite Galois extension over which $f(x)$ splits into linear factors, and define $\Gamma:=\gal(L/\Q$). Let $\alpha_1,\ldots,\alpha_n$ be the roots of $f$ in $L$, and for $1\leq i\leq n$ set $\Gamma_i=\gal(L/\Q(\alpha_i))\subset\Gamma$.

\begin{lemma}
\label{lem1}
Let $p$ be a rational prime unramified in $L$ that is coprime to the denominators of coefficients of $f$. Then $f(x)$ has a root modulo $p$ if and only if the Frobenius conjugacy class $\phi_p\subset \Gamma$ is contained in
\[
S_1:=\bigcup_i \Gamma_i.
\]
\end{lemma}
\begin{proof}
There is a root of $f(x)$ in $\Fp$ if and only if for some (equivalently, every) prime $\mathfrak{P}$ of $L$ over $p$, there is some $i$ for which $\alpha_i\mod \mathfrak{P}$ is in $\Fp\subset\cO_L/\mathfrak{P}$. Now, $\alpha_i\mod \mathfrak{P}$ is in $\Fp$ if and only if $\phi_\mathfrak{P}(\alpha_i)=\alpha_i$, which happens if and only if $\phi_\mathfrak{P}\in\Gamma_i$. So $f$ has a root in $\Fp$ if and only if there exists $\mathfrak{P}|p$ with $\phi_\mathfrak{P}\in\bigcup\Gamma_i$. Since $\bigcup\Gamma_i$ is closed under conjugation, this is equivalent to the condition that $\phi_p\subset\bigcup\Gamma_i$.
\end{proof}

We also need the following fact.
\begin{lemma}
\label{lem2}
Define a set
\[
S_2:=\bigcup_i\big\{\sigma\in \Gamma:C_\Gamma(\sigma)\subset \Gamma_i\big\},
\]
where $C_\Gamma(\sigma)$ is the centralizer of $\sigma$ inside $\Gamma$. The for every $g\in A(L)$, we have
\begin{equation}
\label{eqsett}
\big\{\sigma\in \Gamma:f(g(\sigma))=0\big\}\subseteq S_2,
\end{equation}
and there exists $g\in A(L)$ for which \eqref{eqsett} is an equality of sets.
\end{lemma}
\begin{proof}
If $f(g(\sigma))=0$, then $g(\sigma)=\alpha_i$ for some $i$. By \eqref{eqg}, $g(\sigma)$ is fixed by $C_\Gamma(\sigma)$, so we must have $C_\Gamma(\sigma)\subset \gal(L/\Q(\alpha_i))=H_i$. This proves the containment \eqref{eqsett}. To show that we can choose $g\in A(L)$ for which \eqref{eqsett} is equality, let $\sigma_1,\ldots,\sigma_k\in \Gamma$ be a system of conjugacy class representatives. For $1\leq j\leq k$, if there does not exist $i$ for which $C_\Gamma(\sigma_j)\subset H_i$, then define $g$ to be $0$ on the conjugacy class of $\sigma_j$. If there does exist $i$, the define $g$ on the conjugacy class of $\sigma_j$ by
\[
g(\tau\sigma_j\tau^{-1})=\tau(\alpha_i).
\]
We have $g\in A(L)$ and $\{\sigma:f(g(\sigma))=0\}=S_2$.
\end{proof}

We also need a group-theoretic fact about wreath products.
\begin{lemma}
\label{lemgroup}
Let $\Gamma$ and $A$ be finite groups, with $A$ abelian, and consider the wreath product
\[
\Gamma':=A^\Gamma\rtimes\Gamma.
\]
Let $\pi:\Gamma'\to\Gamma$ be the projection. Then at least
\[
\lp1-\frac{|\Gamma|^2}{|A|}\rp\big|\Gamma'\big|
\]
elements $\xi\in\Gamma'$ satisfy
\begin{equation}
\label{eqx}
\pi\big(C_{\Gamma'}(\xi)\big)\subset\langle\pi(\xi)\rangle.
\end{equation}
\end{lemma}
\begin{proof}
We identify elements of $\Gamma'$ with pairs $(\varphi,\sigma)$, where $\varphi:\Gamma\to A$ and $\sigma\in \Gamma$. Under this identification, multiplication in $\Gamma'$ is given by
\[
(\varphi,\sigma)\circ(\psi,\tau)=\big(\varphi+\psi\circ R_{\sigma},\sigma\tau\big)
\]
(here $R_{\sigma}:\Gamma\to \Gamma$ is right multiplication by $\sigma$). A direct computation shows that $(\varphi,\sigma)$, $(\psi,\tau)\in \Gamma'$ commute if and only if $\sigma$ and $\tau$ commute and
\begin{equation}
\label{eqcom}
\varphi-\varphi \circ R_\tau=\psi-\psi\circ R_\sigma.
\end{equation}
For $\eta:\Gamma\to A$, there exists $\psi:\Gamma\to A$ with $\eta=\psi-\psi\circ R_\sigma$ if and only if
\begin{equation}
\label{eqsumo}
\sum_{n=0}^{ord(\sigma)-1} \eta \circ R_{\sigma^n}=0.
\end{equation}
Combining \eqref{eqcom} and \eqref{eqsumo}, we see that, if $\varphi$, $\sigma$, and $\tau$ are fixed, then there exists $\psi$ such that $(\varphi,g)$ and $(\psi,h)$ commute if and only if
\begin{equation}
\label{eqEcom}
\sum_{n=0}^{ord(\sigma)-1}\big( \varphi\circ R_{\sigma^n\tau}-\varphi \circ R_{\sigma^n}\big)=0.
\end{equation}

For each $\sigma$, $\tau\in\Gamma$, define a group homomorphism
\begin{align*}
\chi_{\sigma,\tau}:A^\Gamma&\to A,\\
\varphi&\mapsto\sum_{n=0}^{ord(\sigma)-1}\big( \varphi(\sigma^n\tau)-\varphi (\sigma^n)\big).
\end{align*}
If $\tau\not\in\langle \sigma\rangle$, then the elements $\sigma^n$ and $\sigma^n\tau$ are all distinct. In this case $\chi_{\sigma,\tau}$ is seen to be surjective, and the kernel of $\chi_{\sigma,\tau}$ has index $|A|$ in $A^\Gamma$. It follows from \eqref{eqEcom} that, for fixed $\tau$, $\sigma$ with $\tau\not\in\langle\sigma\rangle$, there are at most $|A|^{|\Gamma|-1}$ functions $\varphi :\Gamma\to A$ for which $\tau\in\pi(C_{\Gamma'}((\sigma,\varphi))$. Taking the union over all $\sigma$, $\tau\in\Gamma$ with $\tau\not\in\langle\sigma\rangle$, we find that the number of elements $\xi\in \Gamma'$ for which \eqref{eqx} does \emph{not} hold is at most
\[
|\Gamma|^2|A|^{|\Gamma|-1}.
\]
This completes the proof.
\end{proof}

We are now ready to prove Theorem \ref{th3}.
\begin{proof}[Proof of Theorem \ref{th3}]
Suppose $f(x)\in\Q[x]$. It is obvious that
\begin{equation}
\label{eqIN}
\delta\bigg(\big\{p:f(a_p)\equiv 0\!\!\mod p\big\}\bigg) \leq \delta\bigg(\big\{p:f\text{ has a root mod } p\big\}\bigg)
\end{equation}
for all $(a_p)_p\in\P$. We need to show that the inequality \eqref{eqIN} is strict if $f(x)$ has no rational roots, and that we can choose $(a_p)_p\in\P$ to make \eqref{eqIN} arbitrarily close to an equality.

Suppose $f(x)$ has no rational roots. For $(a_p)_p\in\P$, let $L/\Q$ and $g\in A(L)$ be such that $a_p\equiv g(\phi_p)\mod p$, and let $\Gamma\supset S_1\supset S_2$ be as in the statements of Lemmas \ref{lem1} and \ref{lem2}. By the Chebotarev density theorem,
\begin{equation*}
\label{eqd1}
 \delta\lp\big\{p:f\text{ has a root modulo } p\big\}\rp=\frac{\#S_1}{\#\Gamma},
\end{equation*}
\begin{equation*}
\label{eqd2}
\max_{g\in A(L)}\delta\lp\big\{p:f(g(\phi_p))\equiv 0\mod p\big\}\rp =\frac{\#S_2}{\#\Gamma}.
\end{equation*}
We get strictness of \eqref{eqIN} because the identity element of $\Gamma$ is in $S_1$ but not in $S_2$.

To show that \eqref{eqIN} is sharp, we pass from $L$ to an extension $L'/L$ with the property that in $\gal(L'/\Q)$, most elements have small centralizers (in a sense to be made precise). For $L'/\Q$ a finite Galois extension containing $L$, write $\Gamma'=\gal(L'/\Q)$ and $\pi:\Gamma'\surj\Gamma$ for the restriction map. Let
\[
\Gamma_i'=\pi^{-1}(\Gamma_i)=\gal(L'/\Q(\alpha_i))\subset \Gamma',
\]
for $i=1,\ldots,n$. If an element $\sigma\in \Gamma'_i$ satisfies
\begin{equation}
\label{eqxx}
\pi\big(C_{\Gamma'}(\sigma)\big)\subset\langle\pi(\sigma)\rangle,
\end{equation}
then $C_{\Gamma'}(\sigma)\subset\Gamma_i'$. We will show that for every $\epsilon>0$, we can choose the $L'/L$ so that \eqref{eqxx} holds for at least $(1-\epsilon)|\Gamma'|$ elements $\sigma$ of $\Gamma'$. This will prove the theorem.

Let $\epsilon>0$ be given, and choose a positive integer $r$ such that 
\[
\frac{|\Gamma|^2}{2^{r}}<\epsilon.
\]
Let $p_1,\ldots,p_r$ be distinct rational primes that split completely in the Hilbert class field of $L$. For each $i$, let $\beta_i\in \cO_L$ be a generator for a (necessarily degree $1$ and principal) prime of $L$ over $p_i$. Let $L'$ be the extension of $L$ obtained by adjoining a square root of $\sigma(\beta_i)$ for all $\sigma\in \Gamma$ and $1\leq i\leq r$. Then $\Gamma'$ is isomorphic to a wreath product
\[
\Gamma'\cong A^\Gamma\rtimes \Gamma,
\]
with $A= (\Z/2)^r$. The result now follows from Lemma \ref{lemgroup}.
\end{proof}

%
%
\section{Periods}
\label{secfrob}
In this section we prove the equivalence of conditions (1) and (3) of Theorem \ref{th1}. We also explain why $\cP_{\!\A}^0$ is analogous to $\Qb$, and we explain how to obtain other analogues of periods inside $\A$.

\subsection{Dimension $0$}
Let $L/\Q$ be a finite Galois extension, with ring of integers $\cO_L$. There is an isomorphism of $\A$-algebras
\[
L\otimes_\Q\A\cong\frac{\prod_p\cO_L/p\cO_L}{\bigoplus_p\cO_L/p\cO_L}.
\]
For each rational prime $p$, the $p$-th power map is a $\Fp$-algebra endomorphism $F_{p,L}$ of $\cO_L/p\cO_L$, which is an automorphism if $p$ is unramified in $L$. 

\begin{definition}
\label{deffa}
The \emph{$\A$-valued Frobenius automorphism} is the $\A$-algebra automorphism  $F_{\A,L}$ of $L\otimes_\Q\A$ induced by $F_{p,L}$ in the $p$-th factor.
\end{definition}
If we choose a basis for $L$ as a $\Q$-vector space, we can represent $F_{\A,L}$ by a square matrix with entries in $\A$, and the $\Q$-span of the matrix entries does not depend on the choice of basis.

\begin{theorem}
\label{thmat}
For each finite Galois extension $L/\Q$, the $\Q$-span of the matrix entries for $F_{\A,L}$ is equal to the set of elements $(g(\phi_p)\mod p)_p\in\A$ for $g\in A(L)$.
\end{theorem}
\begin{proof}
The $\Q$-span of matrix coefficients for $F_{\A,L}$ is  the image of the map
\begin{align}
\label{map}L^\du\otimes_\Q L &\to \A,\\
\varphi\otimes y&\mapsto \big(\varphi(y^p)\mod p\big)_p.\nonumber
\end{align}
Here $L^\du$ is the $\Q$-linear dual of $L$. The trace form induces an isomorphism of $L$ with $L^\du$, so the image of \eqref{map} is equal to the image of
\begin{align*}
L\otimes L&\to \A,\\
x\otimes y&\mapsto \lp\Bigg(\sum_{\sigma\in \Gamma}\sigma(xy^p)\Bigg)\mod p\rp_p,
\end{align*}
where $\Gamma=\gal(L/\Q)$.

It follows from the proof of Theorem \ref{thseqA} that $\{(g(\phi_p)\mod p)_p\}$ is equal to the image of the map
\begin{align*}
(L\otimes L)^\Gamma&\to \A,\\
\sum_i x_i\otimes y_i&\mapsto \lp\Bigg(\sum_i x_i y_i^p\Bigg)\mod \mathfrak{P}\rp_p,
\end{align*}
where for each $p$ we have chosen a prime $\mathfrak{P}$ of $L$ over $p$. The result now follows from the fact that
\begin{align*}
L\otimes L&\to (L\otimes L)^\Gamma,\\
x\otimes y&\mapsto\sum_\sigma \sigma(x)\otimes\sigma(y)
\end{align*}
is surjective.
\end{proof}

The algebraic de Rham cohomology of $\spe(L)$ (which we view as a $0$-dimensional algebraic variety over $\Q$) is identified with $L$. Thus $\P$ is the $\Q$-span of the matrix coefficients for the isomorphism
\[
H^0_{dR}(\spe(L))\otimes\A\iso H^0_{dR}(\spe(L))\otimes\A,
\]
for $L$ ranging over the finite Galois extensions of $\Q$. If instead we look at de Rham-Betti comparison isomorphism
\[
H^0_{dR}(\spe(L))\otimes\C\iso H^0_{B}(\spe(L))\otimes\C
\]
for varying $L$, the $\Q$-span of the matrix coefficients is $\Qb$. For this reason $\P\subset\A$ is analogous to $\Qb\subset\C$. By contrast, the integral closure of $\Q$ inside $\A$ is uncountable.

\subsection{Positive dimension}
\label{ssgen}
The characterization of $\P$ as matrix coefficients of the $\A$-valued Frobenius can be generalized to produce elements of $\A$ from varieties of positive dimension. If $X$ is a variety defined over $\Q$ and $i\geq 0$ is an integer, the algebraic de Rham cohomology $H_{dR}^i(X)$ is finite-dimensional vector space over $\Q$, and for all sufficiently large $p$ there is a distinguished automorphism
\begin{equation}
\label{eqdefFp}
F_{p,X}:H_{dR}^i(X)\otimes\Q_p\iso H_{dR}^i(X)\otimes\Q_p
\end{equation}
coming from crystalline cohomology (see \cite{Ked09}). Matrix coefficients for $F_{p,X}$ with respect to a $\Q$-basis are (one type of) $p$-adic periods of $X$. Each matrix coefficient for $F_{p,X}$ is $p$-integral for all sufficiently large $p$, so reduction modulo $p$ (for all large $p$ at once) gives an element of $\A$. These elements are called $\A$-valued periods in \cite{Ros18a}. It is convenient to assemble the maps $F_{p,X}$ to form an $\A$-valued Frobenius map
\[
F_{\A,X}:H_{dR}^i(X)\otimes\A\to H_{dR}^i(X)\otimes\A,
\]
whose matrix coefficients are $\A$-valued periods (the map $F_{\A,X}$ is no longer an isomorphism).
Details can be found in \cite{Ros18a}, \S6.

If we instead use the de Rham-Betti comparison isomorphism
\[
comp_X:H^i_{dR}(X)\otimes\C\iso H^i_{B}(X)\otimes\C,
\]
matrix coefficients are the ordinary (complex) periods of $X$. So in this analogy $\A$ corresponds to $\C$, and $F_{\A,X}$ corresponds to $comp_X$. Define $\cP_\A\subset\A$ (resp.\ $\cP_\C\subset\C$) to be the $\Q$-span of the matrix coefficients for $F_{\A,X}$ (resp.\ $comp_X$), as $X$ ranges through all varieties over $\Q$. By taking $X$ to have dimension $0$ we see that $\P\subset\cP_\A$ and $\Qb\subset\cP_\C$.

The period conjecture of Grothendieck (see \cite{And04}, \S7.5) would imply that there is a $\Q$-algebra homomorphism
\[
\Delta:\cP_\C\to\cP_\C\otimes_\Q\cP_\A.
\]
Concretely, fix a variety $X$ and bases for $H^i_{dR}(X)$ and $H^i_{dR}(X)$, say of length $n$. Write $F_{\A,X}$ and $comp_X$ as matrices $(\alpha_{i,j})\in M_n(\A)$ and $(\beta_{i,j})\in M_n(\C)$, respectively. The map $\Delta$ is then given by
\begin{equation}
\label{eqDelta}
\Delta(\beta_{i,j}) = \sum_{k=1}^n \beta_{i,k}\otimes\alpha_{k,j}\in\cP_\C\otimes_\Q\cP_\A.
\end{equation}
A priori the right hand side of \eqref{eqDelta} might depend on $X$, $i$, and $j$, but the the period conjecture implies that in fact the right hand side depends only on the value $\beta_{i,j}\in\cP_\C$.

Every algebraic number occurs as a matrix coefficient for $comp_X$ for some $0$-dimensional $X$. Since the $\A$-valued periods of this $X$ are in $\P$, this implies $\Delta$ takes $\Qb\subset\cP_\C$ into $\P\otimes_\Q\cP_\C$. So the truth of the period conjecture would imply that if we see an algebraic number as a complex period of an arbitrary variety, we will also see elements of $\P$ in the $\A$-valued periods of that variety.

\ack We thank Jeffrey Lagarias for helpful comments. We thank the anonymous referee for helpful suggestions on the structure of the paper.

\bibliographystyle{hplain}
\newcommand{\noop}[1]{} \def\cprime{$'$}

\end{document}